\documentclass[pdflatex,sn-mathphys-num]{sn-jnl}

\usepackage{graphicx}%
\usepackage{multirow}%
\usepackage{amsmath,amssymb,amsfonts}%
\usepackage{amsthm}%
\usepackage{mathrsfs}%
\usepackage[title]{appendix}%
\usepackage{xcolor}%
\usepackage{textcomp}%
\usepackage{manyfoot}%
\usepackage{booktabs}%
\usepackage{algorithm}%
\usepackage{algorithmicx}%
\usepackage{algpseudocode}%
\usepackage{listings}%
\def\R{{\Bbb R}}
\def\BB{{\Bbb B}}
\def\C{{\Bbb C}}
\def\D{{\Bbb D}}

\def\M{{\Bbb M}}

\def\N{{\Bbb N}}

\theoremstyle{thmstyleone}%
\newtheorem{theorem}{Theorem}
\newtheorem{proposition}[theorem]{Proposition}%
\newtheorem{lemma}[theorem]{Lemma}%

\theoremstyle{thmstyletwo}%

\theoremstyle{thmstylethree}%
\newtheorem{example}{Example}
\newtheorem{remark}{Remark}
\begin{document}

\title[Characterizing Weighted Composition Operators on Weighted-Type High-Order Growth Spaces
]{Characterizing Weighted Composition Operators on Weighted-Type High-Order Growth Spaces via the Component Function $\varphi_p$
}


\author*[1,2]{\fnm{Thai} \sur{Thuan Quang}}\email{thaithuanquang@qnu.edu.vn}

%

\affil*[1]{\orgdiv{Department of Mathematics and Statistics}, \orgname{Quy Nhon University}, \\ \orgaddress{\street{170 An Duong Vuong}, \city{Quy Nhon}, 
\state{Binh Dinh}, \country{Viet Nam}}}

%


\abstract{
Let $\psi$ be a holomorphic function on the open unit ball $\BB \subset \C^N$, and let $\varphi$ be a holomorphic self-map of $\BB$, associated with normal weights $\nu$ and $\mu$. We consider the weighted composition operator $
W_{\psi,\varphi} : \mathcal H_\nu^{(n)} \to \mathcal H_\mu^{(m)}, \quad n,m \in \N,$
acting between weighted-type high-order growth spaces. Unlike previous studies that involve the full symbol $\varphi$, this paper establishes characterizations of the boundedness, compactness, and asymptotic norm estimates of $W_{\psi,\varphi}$ \emph{solely in terms of the symbol $\psi$ and a single component function $\varphi_p$ of $\varphi$}, offering a new approach to the analysis of such operators.
}




\maketitle

\section{Introduction} \label{sec_1} 

Let $\mathbb{B} = \{z\in \BB: |z| =\sqrt{\langle z, z\rangle}< 1\}$ denote the open unit ball in $\mathbb{C}^N$. A continuous radial weight 
$\omega:\mathbb{B}\to(0,\infty)$ is called \emph{normal} if there exist constants 
$0\le \delta <1$ and $0<a<b<\infty$ such that
\begin{equation}\label{w_1}
 \frac{\omega(t)}{(1 - t)^a} \ \text{is decreasing on $[\delta, 1)$},\quad \lim_{t\to1}\frac{\omega(t)}{(1 - t)^a}=0, \tag{$ W_1 $}
\end{equation}
\begin{equation}\label{w_2}
	\frac{\omega(t)}{(1 - t)^b} \ \text{is increasing on $[\delta, 1)$},\quad \lim_{t\to1}\frac{\omega(t)}{(1 - t)^b} =\infty. \tag{$ W_2 $}
\end{equation}

Such weights are fundamental in controlling boundary growth of holomorphic functions.

For $n\in \mathbb{N}_0$, the weighted-type $n$-order growth space associated with a normal weight $\omega$ is defined by
\[
\mathcal H_\omega^{(n)} := \Big\{ f\in H(\mathbb{B}) : \|f\|_{H_\omega^{(n)}} 
= |f(0)| + \sup_{z\in\mathbb{B}} \omega(z) \, |R^{(n)} f(z)| < \infty \Big\},
\]
where $R^{(n)}$ denotes the $n$-th radial derivative. This class of spaces was introduced and systematically studied by T.~T.~Quang \cite{Qu1}, who also initiated the study of composition operators acting on them. These spaces provide a natural high-order extension of classical growth, Bloch-type, and Zygmund-type spaces, forming a flexible framework for operator theory.

Given $\psi \in H(\mathbb{B})$ and a holomorphic self-map 
$\varphi = (\varphi_1,\dots,\varphi_N) \in S(\BB)$ of $\mathbb{B}$, we consider the weighted composition operator
\[
W_{\psi,\varphi}: \mathcal H_\nu^{(n)} \to \mathcal H_\mu^{(m)}, \quad 
 f \mapsto \psi \cdot (f\circ \varphi),
\]
where $\nu, \mu$ are normal weights and $n,m \in \mathbb{N}_0$.

Classical studies have typically characterized boundedness and compactness of $W_{\psi,\varphi}$ in terms of the full self-map symbol $\varphi$. More recently, it has been observed that these properties may be reduced to dependence on a single component function $\varphi_p$ of $\varphi$. While this reduction has been explored in low-order settings such as Bloch- and Zygmund-type spaces (see \cite{LDQ1, LDQ2, LQ}), it has not yet been extended to the broader class of high-order growth spaces.

The present paper advances this direction. We show that the boundedness, compactness, and asymptotic norm estimates of $W_{\psi,\varphi}$ can be characterized entirely in terms of the symbol $\psi$ and one component $\varphi_p$ of $\varphi$. This provides a novel approach to the study of weighted composition operators and continues the line of investigation initiated in Quang's earlier work.

Outside the Introduction, the paper is organized as follows. Section \ref{sec_2} reviews weighted-type high-order growth spaces, normal weights, and basic preliminaries. Section \ref{sec_3} presents auxiliary results. Section \ref{sec_4} introduces the main conditions on the symbols. Section \ref{sec_5} establishes characterizations of boundedness and compactness. 

	Throughout this paper, we use the notions $a \lesssim b$ and $a \asymp b$ for non negative quantities $a$ and $b$ to mean $a \le Cb$ and, respectively, $b/C \le a \le Cb$ for some inessential constant $C >0.$

\section{Preliminaries}\label{sec_2}

Throughout this paper we will use the following symbols:  
\[\aligned
&|\vec{k}| = k_1 +\cdots + k_j, \quad \vec{k}! = k_1!\cdots k_j!\quad\text{for  multi-indeces}\ \vec{k} =(k_1, \ldots, k_j) \in \N_0^j, \\
&K_{i,j} =  \Big\{\vec{k}  \in  \N^j: \  |\vec{k}|=i\Big\},  \quad K^0_{i,j} =  \Big\{\vec{k}  \in  \N_0^j: \  |\vec{k}|=i\Big\}, \\
&C^n_{\vec{k}} = \binom{n}{k_1,\ldots,k_j} = \frac{n!}{\vec{k}!}\quad\text{with $\vec{k} \in K^0_{i,j}$}, \\
&L_{j}  = \Big\{\vec{l} = (l_1, \ldots, l_j) \in  \N^j: \  1\le \l_1, \ldots, l_j \le N\Big\}, \\  
&\frac{\partial^j}{\partial z_{\vec{l}}} = \frac{\partial^j}{\partial z_{l_1}\cdots\partial z_{l_j}} \quad\text{with $\vec{l} \in L_{j}$}.
\endaligned\]

For a normal weight $\omega $ on $\BB, $ for every $z \in \BB,$ we denote
\[
  I^k_\omega(z) :=\int_0^{|z|}\int_0^{t_{k-1}}\cdots\int_0^{t_1}\frac{1}{\omega(t)}dtdt_1\ldots dt_{k-1},  \quad   k \ge 1.\]

From the definition of radial derivative and a direct calculation, the following so-called \textit{Newton-Leibniz formula} holds for every $\psi, f \in H(\BB)$ and $\varphi=(\varphi_1, \ldots, \varphi_N) \in S(\BB):$
 \[R^{(n)}(\psi \cdot(f \circ \varphi))(z)=\sum_{i=0}^n\binom{n}{i} R^{(n-i)}\psi(z)R^{(i)}(f \circ \varphi)(z), \quad z \in \BB,\] 
 where \begin{equation}\label{eq_Bruno}
R^{(n)}(f \circ \varphi)(z) =\sum_{j=1}^n \sum_{\vec{l} \in L_{j}} \left(\frac{\partial^j f(\varphi(z))}{\partial z_{\vec{l}}}
\sum_{\vec{k} \in K_{n,j}} C_{\vec{k}}^{n} \prod_{t=1}^j R^{(k_t)} \varphi_{l_t}(z)\right).\end{equation} 
The formula \eqref{eq_Bruno} obtained from \cite{HJ}.

\subsection{Weighted-Type High-Order Growth Spaces}\label{sec_3}
In the subsection, we review some fundamental aspects of weighted high-order growth spaces and certain computational formulas that have been established in \cite{Qu1}.

 Let $\omega$ be a normal weight on $\BB.$  
 For $n \in \N_0 = \N \cup \{0\},$ we can easily verify that the   weighted-type $n$-order growth space
 \[\mathcal{H}^{(n)}_{\omega} := \bigg\{f \in H(\BB): \|f\|_{\mathcal{H}^{(n)}_{\omega}} := |f(0)|+\|f\|_{s\mathcal{H}^{(n)}_{\omega}}<\infty\bigg\}\]
 is a Banach space when equipped with the norm 
 $\|\cdot\|_{\mathcal{H}^{(n)}_{\omega}},$ 
where
\[\|f\|_{s\mathcal{H}^{(n)}_{\omega}} := \sup_{z \in \BB}\omega(z)|R^{(n)}f(z)|,\]
where $R^{(0)}f = f, R^{(1)}f(\cdot) = \langle \cdot, \nabla f(\cdot)\rangle, R^{(n)}f = R(R^{(n-1)}f). $

For each $n \ge 0$, the space $\mathcal H_\omega^{(n+1)}$ is contained in $\mathcal H_\omega^{(n)}$. Moreover, the norm 
$\|\cdot\|_{\mathcal H_\omega^{(n)}}$ on $\mathcal H_\omega^{(n+1)}$ is  controlled by the norm 
$\|\cdot\|_{\mathcal H_\omega^{(n+1)}}$; in other words, there exists a constant $C>0$ such that
\begin{equation}\label{compair_norm}
\|f\|_{\mathcal H_\omega^{(n)}} \le C \, \|f\|_{\mathcal H_\omega^{(n+1)}}, \quad   f \in \mathcal H_\omega^{(n+1)}.
\end{equation}
Furthermore, if $f \in \mathcal H_\omega^{(n)}$, then for every $k \le n$, the $k$-th radial derivative 
$R^{(k)} f$ belongs to the space $\mathcal H_\omega^{(n-k)}$ \cite[Proposition 3.1]{Qu1}.


  For each $z \in \BB,$   we define \textit{the point-evaluation functional $ \delta^{\mathcal{H}^{(n)}_{\omega}}_z $ at $ z $} as follows: 
\[\delta^{\mathcal{H}^{(n)}_{\omega}}_z(f) := f(z),\quad f \in \mathcal{H}^{(n)}_{\omega}.\]

\begin{proposition}[\cite{Qu1}, Proposition 3.3]\label{exam_E_delta} We have  the following estimates for the point evaluation functional:
	\[\Big\|\delta^{\mathcal{H}^{(0)}_{\omega}}_z\Big\| = \dfrac{1}{\omega(z)};\quad  \Big\|\delta^{\mathcal{H}^{(n)}_{\omega}}_z\Big\| \asymp 1+I^n_\omega(z), \quad n\ge 1.\]
 \end{proposition}
  For every $f \in \mathcal{H}^{(n)}_\omega$ and every $z\in \BB,$ the following is true for every $\vec{l} \in L_j,$ (see \cite[Theorem 3.3]{Qu1}: 
\begin{equation}\label{eq_est_parial_f}
\aligned
&\bigg|\frac{\partial^j f(z)}{\partial z_{\vec{l}}}\bigg| 
\lesssim \Big\|\delta_z^{\mathcal{H}^{(n-j)}_\omega}\Big\|\|f\|_{\mathcal{H}^{(n)}_\omega}\lesssim \Big\|\delta_z^{\mathcal{H}^{(n)}_\omega}\Big\|\|f\|_{\mathcal{H}^{(n)}_\omega}, \quad   \text{for}\  j = 0, \ldots, n; \\
&\bigg|\frac{\partial^{n+k} f(z)}{\partial z_{\vec{l}}}\bigg| 
\lesssim \frac{1}{\omega(z)(1-\|z\|^2)^k}\|f\|_{\mathcal{H}^{(n)}_\omega}, \quad  \text{for}\ k = 1, 2, \ldots. \endaligned
\end{equation}
 
%
 The following inequality is extracted from estimates (4.3) in \cite{Qu1}:
 \begin{equation}\label{eq_est_parial_f_more2}
 \mu(z)|R^{(n)}W_{\psi,\varphi}(f)|  \le \sum_{j=0}^n\bigg|\frac{\partial^j f(\varphi(z)) }{\partial z_{\vec{l}}}
 \bigg| \mu(z)\bigg|\sum_{i=j}^n\binom{n}{i}R^{n-i}(\psi(z)) \mathfrak B_{i,j}(R\varphi(z))\bigg|. \\
\end{equation}

%


\subsection{The set $\widetilde S_p(\BB) $ }
%
  In this section we recall the set 
$S^*_p(\BB)$ (see \cite{LDQ1}) and simultaneously introducing the set  $\widetilde S_p(\BB)$
 together with some of its geometric characterizations. We denote
	\[\aligned
	S^*_p(\BB) &= \Big\{\varphi \in S(\BB): \ \varphi(\BB) \supseteq \D_p := \big\{\lambda e_p:\ \lambda \in \D\big\}\Big\}, \\
	\widetilde S_p(\BB) &= \Big\{\varphi \in S(\BB):\ 0\in \varphi(\BB),    \forall z \in \BB, \exists z' \in \BB: |\varphi_p(z')| = |\varphi(z)|\Big\}.
	\endaligned\]
for  $p\in \{1, \ldots, N\}.$
  
  Geometrically speaking, for $\varphi \in S^*_p(\BB),$ the image $\varphi(\BB)$        covers the entire unit disc in the $p$-th coordinate plane. Meanwhile, it is easy to see that
  \[\widetilde S_p(\BB) = \bigg\{\varphi \in S(\BB):\ 0\in \varphi(\BB),  \sup_{x \in \varphi_p(\BB)}|x| = \sup_{z \in  \varphi(\BB)}|z|\bigg\}.\]
  Indeed, it is obvious that $\sup_{x\in \varphi_p(\BB)}|x| \le \sup_{z\in \varphi(\BB)}|z|.$ Conversely, for every $z\in \varphi(\BB)$
  there exists $z'\in \BB$ such that $|\varphi(z)| = |\varphi_p(z')|\le \sup_{x\in \varphi_p(\BB)}|x|.$ This implies that for every $\varphi \in \widetilde S_k(\BB)$, the image $\varphi(\BB)$ necessarily contains the origin, and the maximal distance from $0$ to the boundary remains unchanged when projected onto $\D_k$. Consequently, we obtain the strict inclusion $\widetilde S_k(\BB) \supsetneq S_k^*(\BB)$.
  

Examples are given below to support this inclusion.

\begin{example} In $\C^2,$ consider $\BB\subset \C^2,$ and the function $\varphi: \BB \to \C^2$ given by
\[\varphi(z_1, z_2) = \frac{1}{2}\Big(z_1+\frac{1}{2}, \frac{z_2}{2}\Big).\]
Obviously,   $\varphi \in  S(\BB)$ and $0 \in \varphi(\BB).$  For every $z = (z_1, z_2) \in \BB,$ consider $x\in \D_1$ such that $|x| = |\varphi(z)|.$ Then, it is easy to check that for $z' = (2x - \frac{1}{2}, z'_2) \in \BB$ we have $\varphi_1(z') = x,$ consequently, $|\varphi_1(z')| = |\varphi(z)|.$
 Thus, $\varphi \in \widetilde S_1(\BB). $

However, it is easy to check that the point $-\frac{i}{2} \in \D_1 \setminus \text{\rm Pr}_{\D_1}\varphi(\BB). $ This means   $\varphi \notin S^*_1(\BB).$
\end{example}
\begin{example} In $\C^3,$ consider $\varphi =(\varphi_1, \varphi_2, \varphi_3) \in S(\BB)$ given by

\[\varphi(z_1, z_2, z_3)= \frac{1}{5}\Big(3z_1, \frac{3}{2}z^2_2+i, z_3\Big).\]

Obviously,    $\varphi \in  S(\BB)$ and $0 \in \varphi(\BB).$   For every $z = (z_1, z_2,z_3) \in \BB,$ consider $y\in \D_2$ such that $|y| = |\varphi(z)|.$ Then, it is easy to check that for $z' = \big(z'_1, \sqrt{\frac{2(5y-i)}{3}},z'_3\big) \in \BB$ we have $\varphi_2(z') = y,$ consequently, $|\varphi_2(z')| = |\varphi(z)|.$
 Thus, $\varphi \in \widetilde S_2(\BB). $

However, it is easy to check that the point $-\frac{i}{2} \in \D_2 \setminus {\rm Pr}_{\D_2}\varphi(\BB). $ This means   $\varphi \notin S^*_2(\BB).$
\end{example}

\begin{example} Fix $\alpha_1, \ldots, \alpha_N \in \R$ and $a=(a_1, \ldots, a_N) \in \BB \subset \C^N$ such that $|a_p| = \max_{1\le k\le N}|a_k|,$ $p \in \{1, \ldots, N\}.$ Consider $\varphi =(\varphi_1, \ldots. \varphi_N) \in S(\BB)$ given by
\[\varphi(z) = (e^{i\alpha_1}a_1z_1, \ldots, e^{i\alpha_N}a_Nz_N), \quad z= (z_1, \ldots, z_N).\]
It is easy to see that $\varphi(\BB)\subset \BB$ and $\varphi(0)=0.$  We have 
\[|\varphi(z)|\le \max_{1\le k \le N}|a_k||z| = |a_p||z| < |a_p| = |e^{i\alpha_p}a_p|,\quad z\in \BB.\]
On the other hand, $\varphi_p(\BB) =\{x \in \C:\ |x| \le |a_p|\}, $
This means $\sup_{x \in \partial(\text{\rm Pr}_{\D_p}\varphi(\BB))}|x| \ge \sup_{z\in \partial\varphi(\BB)}|z|.$ Therefore, $\varphi \in \widetilde S_p(\BB).$

While, it is clear that $\varphi \notin S^*_p(\BB)$ because $\{x \in \C: \ |x|> |a_p|\} = \D_N \setminus \text{\rm Pr}_{\D_p}\varphi(\BB) \neq \varnothing.$   \end{example}

\section{Auxiliary Results}\label{sec_3}
By $Aut(\BB)$ we denote  the automorphism group of $\BB$ that  consists of all bi-holomorphic mappings of $\BB.$ It is known that every  $\gamma \in Aut(\BB)$ is a unitary transformation of $\C^N$ if and only if $\gamma(0) = 0$ (see \cite[Lemma 1.1]{Zh}).	
		For any $ \alpha \in \BB \setminus \{0\}, $ we define 
\begin{equation}\label{eq_Mobius}
	\gamma_\alpha(z) = \frac{\alpha - P_\alpha(z) - s_\alpha Q_\alpha(z)}{1 - \langle z, \alpha\rangle}, \quad z \in \BB,
\end{equation} 
where $s_\alpha = \sqrt{1 - |\alpha|^2},$   $P_\alpha(z) = \frac{\langle z, \alpha\rangle}{|\alpha|^2}\alpha,$ 

When $\alpha = 0,$ we simply define $\gamma_\alpha(z)=-z.$ 
It is obvious that each $\gamma_\alpha$  is a holomorphic mapping from $\BB$ into $\C^N.$ It is well known that each $ \gamma_\alpha $ is a homeomorphism of the closed unit ball $ \overline{\BB} $ onto $ \overline{\BB} $ and every automorphism $ \gamma $ of $ \BB $ is the form $ \gamma = \gamma_\alpha U, $ where $ U $ is a unitary transformation of $ \C^N. $	

It is known that
\begin{equation}\label{eq_modul_phi}
1 - |\gamma_\alpha(z)|^2 = \frac{(1-|\alpha|^2)(1 -|z|^2)}{|1-\langle z, \alpha\rangle|^2}.
\end{equation}

\begin{lemma}
\label{lem_B_n} Let $ \nu $ be  a  normal  weight   on $ \mathbb B, $  and $  h $ be a  positive,  real-valued bounded  function   defined on $ \mathbb B$ satisfying $\lim_{|z|\to 1}h(z) > 0.$  
\begin{enumerate}[\rm(a)]
		\item[\rm(a)] Assume that    $\varphi = (\varphi_1, \ldots, \varphi_N) \in S^*_p(\BB)$ for some $p \in \{1, \ldots, N\},$ $\varphi(0) = 0, $ and
		\[\M^j_p := \sup_{w\in \BB}h(w)\Big\|\delta^{\mathcal H^{(j)}_\nu}_{\varphi_p(w)}\Big\| <\infty,\quad j=1, 2,\ldots,  \]
then there exists constant $C_j>0$ such that 
	\begin{equation}\label{lem_S*1}  \sup_{z \in \BB}h(z)\Big\|\delta^{\mathcal H^{(j)}_\nu}_{\varphi(z)}\Big\| \le C_j\M^j_p. 
	\end{equation}
	\item[\rm(b)] Let $\alpha \in \mathbb B \setminus \{0\},$   $\gamma= \gamma_\alpha = (\gamma_1, \ldots, \gamma_N)$ be defined as in  \eqref{eq_Mobius}. For every $p \in \{1, \ldots, N\}$ satisfying 
	\[\M^{j*}_p := \sup_{w\in \mathbb B}h(w)\Big\|\delta^{\mathcal H^{(j)}_\nu}_{w_p}\Big\| <\infty,\quad j =1, 2,\ldots,\] 
	there exists $C_{p,j} >0$ such that
		\begin{equation}\label{lem_gamma*1} \sup_{z\in \BB}h(z)\Big\|\delta^{\mathcal H^{(j)}_\nu}_{\gamma_p(z)}\Big\| \le C_{p,j}\M^{j*}_p.
		\end{equation}
		\item[\rm(c)] For every $p \in \{1, \ldots, N\},$  we have  		\[\sup_{z\in \mathbb B}\frac{\nu(z)}{\nu(\gamma(z))} < \infty,\quad \sup_{z\in \mathbb B}\frac{\nu(z_p)}{\nu(\gamma_p(z))} < \infty.\]
	\end{enumerate}
 \end{lemma}
 \begin{proof} Let $j \in \N$ be fixed. 
Obviously, \eqref{lem_S*1} holds when $|\varphi(z)|=0.$ It suffices to consider the case $|\varphi(z)|>0.$

Since $\lim_{|z|\to1}h(z) > 0$ we can find   $\delta_0 \ge \delta$  such that $\inf_{|z| \ge \delta_0}h(z)>0.$ Then, by the boundedness of $h$ we have $C^+_{\delta_0} := \frac{\sup_{|z|\ge \delta_0}h(z)}{\inf_{|z| \ge \delta_0}h(z)} <\infty.$  
Denote $C^-_{\delta_0} := \frac{\delta_0}{m_{\nu,\delta_0}}\sup_{w\in \BB}h(w)$ where $m_{\nu,\delta_0} = \inf_{|z|\le \delta_0}\nu(z) >0.$

(a) Fix $z \in \BB.$ First, note that,  since   $\M^j_p<\infty,$  and  $h$ is bounded on $\BB,$ there exists $C^-_j >0$ such that $C^-_{\delta_0}  \le C^-_j\M^j_p.$

In the case $0<|\varphi(z)| \le \delta_0$ we have
\[\aligned
h(z)I^j_\nu(\varphi(z)) &\le h(z)I^j_\nu(\delta_0) 
\le \sup_{w\in \BB}h(w)\int_0^{\delta_0}\cdots\int_0^{\delta_0} \frac{dt_1\ldots dt_j}{m_{\nu,\delta_0}} \\
&\le \frac{\delta^j_0}{m_{\nu,\delta_0}}\sup_{w\in \BB}h(w) \le \frac{\delta_0}{m_{\nu,\delta_0}}\sup_{w\in \BB}h(w) 
= C^-_{\delta_0} \le C^-_j\M^j_p.
\endaligned\]

Now we consider  the case $|\varphi(z)| \ge \delta_0.$ By $\varphi \in S^*_p(\BB), $ there exists $z' \in \BB$ such that $|\varphi_p(z')| = |\varphi(z)|.$ Since $\varphi(0)=0,$ we have $|z'| \ge |\varphi(z')| \ge |\varphi_p(z')| = |\varphi(z)| > \delta_0.$ Therefore, $h(z') \ge \inf_{|z|\ge \
\delta_0}h(z).$ Then, since $\nu$ is decreasing on $[\delta, 1),$ we get the following estimates 
\[\aligned
&h(z)I^j_\nu(\varphi(z))  = h(z) \int_0^1\ldots\int_0^1\frac{|\varphi(z)|^jdt_1\cdots dt_j}{\nu(t_1\cdots t_j\varphi(z))}\\
&=h(z) \int_0^1\ldots\int_0^1\int_0^{\delta_0/t|\varphi(z)|}\frac{|\varphi(z)|^jdt_1\cdots dt_j}{\nu(t_1\cdots t_j\varphi(z))} + h(z)\int_0^1\ldots\int_0^1\int_{\delta_0/t|\varphi(z)|}^1\frac{|\varphi(z)|^jdt_1\cdots dt_j}{\nu(t_1\cdots t_j\varphi(z))}\\
&\le \sup_{w\in \BB}h(w)\int_0^1\ldots\int_0^1\int_0^{\delta_0} \frac{dt_1\cdots dt_j}{m_{\nu, \delta_0}}+\frac{h(z)}{h(z')}h(z') \int_0^1\ldots\int_0^1\int_{\delta_0/t|\varphi_k(z')|}^1 \frac{|\varphi_p(z')|^jdt_1\cdots dt_j}{\nu(t_1\cdots t_j\varphi_p(z'))}\\
&\le C^-_j\M^j_p + C^+_{\delta_0} \sup_{w\in \BB}h(w)I^j_\nu(\varphi_p(w))  = (C^-_j +C^+_{\delta_0})\M^j_p <\infty.\endaligned\]

Combining this with the boundedness of the function $h$ and Proposition \ref{exam_E_delta}, we conclude that \eqref{lem_S*1} holds.
 
(b) First, recall from \cite{Zh} that $P_\alpha $ is the orthogonal projection from $\mathbb C^N$ onto the one dimensional subspace $[\alpha ]$ generated by $\alpha $ and  $Q_\alpha $ is the orthogonal projection from $\mathbb C^N$
	onto $\mathbb C^N \ominus [\alpha ].$ Note that  $\gamma(\alpha) =0$ and $(\gamma \circ \gamma)(z) = z$ for every $z \in \mathbb B.$ 

Let $p \in \{1, \ldots, N\}$ be fixed, with $\M^*_p< \infty.$  Since $\gamma$ is surjective, for every $z\in \BB,$ there exists $z' \in \BB$ such that
	\[\gamma(z') = (0, \ldots, \gamma_p(z), 0, \ldots, 0).\]
It is clear  that $\gamma_j(z')=0$ for $j \in \{1, \ldots, N\}\setminus \{p\}$ and
\[\gamma_p(z) =\gamma_p(z') = \frac{\alpha _p - P_{\alpha, p}(z') - s_\alpha Q_{\alpha, p}(z')}{1 - \langle z',\alpha \rangle},\]
	where $P_{\alpha, p}(z') := \frac{\langle z',\alpha \rangle}{|\alpha|^2}\alpha_p,$ and $Q_{\alpha, p}(z') = z'_p - \frac{\langle z',\alpha \rangle}{|\alpha|^2}\alpha_p.$ 	
	
		Since $\alpha  - P_\alpha (z)$ and $Q_\alpha (z)$ are perpendicular in $\mathbb C^N,$  	we have
	\[\aligned
	&|\alpha_p - P_{\alpha, p}(z') - s_\alpha Q_{\alpha, p}(z')|^2  =|\alpha_p - P_{\alpha, p}(z')|^2 + (1-|\alpha|^2)\big(|z'_p|^2 - |P_{\alpha, p}(z')|^2\big)\\
	&\ = |\alpha_p|^2\bigg(1 - 2\text{Re}\frac{\langle z',\alpha\rangle}{|\alpha|^2} +\frac{|\langle z',\alpha\rangle|^2}{|\alpha|^4}\bigg) - \frac{|\langle z',\alpha\rangle|^2}{|\alpha|^4}|\alpha_p|^2 + \frac{|\langle z',\alpha\rangle|^2}{|\alpha|^2}|\alpha_p|^2 +(1 - |\alpha|^2)|z_p|^2 \\
	&\ = \frac{|\alpha_p|^2}{|\alpha|^2}\bigg(|\alpha|^2-1 +1 - 2\text{Re}\langle z',\alpha\rangle + |\langle z',\alpha\rangle|^2\bigg) + (1 - |\alpha|^2)|z'_p|^2 \\
	&\ = \frac{|\alpha_p|^2}{|\alpha|^2}(|\alpha|^2-1) + \frac{|\alpha_p|^2}{|\alpha|^2}|1 - \langle z',\alpha\rangle|^2 + (1 - |\alpha|^2)|z'_p|^2. \\
	\endaligned\]
	This yields that 
		\begin{equation}\label{est_gamma_k}
	\aligned
	|\gamma_p(z)|^2 &= |\gamma(z')|^2 = |\gamma_p(z')|^2 =  \frac{|\alpha_p|^2}{|\alpha|^2} - \frac{(1 - |\alpha|^2)\bigg(\frac{|\alpha_p|^2}{|\alpha|^2} -|z'_p|^2\bigg)}{|1 - \langle z',\alpha\rangle|^2} \\
	& \le \frac{|\alpha_p|^2}{|\alpha|^2}  - \frac{(1 - |\alpha|)\bigg(\frac{|\alpha_p|^2}{|\alpha|^2} -|z'_p|^2\bigg)}{1 +|\alpha|} \\
	&= \frac{2|\alpha_p|^2}{|\alpha|(1+|\alpha|)} +\frac{1-|\alpha|}{1 +|\alpha|}|z'_p|^2 = A_p^2 + A^2|z'_p|^2  < A_p^2+A^2,
	\endaligned\end{equation}
	where $A^2_p := 	 \frac{2|\alpha_p|^2}{|\alpha|(1+|\alpha|)} $ 
	and $A^2 := 
	\frac{1-|\alpha|}{1 +|\alpha|}.$ 
	It is easy to check that $A_p^2 + A^2< 1.$  
	Then,  for all $ |z'_p| \ge |\alpha|,$ we obtain the following the estimate
	 		\[\aligned
	I^j_\nu(\gamma_p(z))&\le  \int_0^{A|z'_p|}\int_0^{t_{j-1}}\cdots\int_0^{t_1}\frac{dtdt_1\ldots dt_{j-1}}{\nu(t)} \\
	&\quad  +\int_{A|z'_p|}^{\sqrt{A_p^2+A^2|z'_p|^2}}\int_0^{t_{j-1}}\cdots\int_0^{t_1}\frac{dtdt_1\ldots dt_{j-1}}{\nu(t)} \\
	&\le \int_0^{|z'_p|}\int_0^{t_{j-1}}\cdots\int_0^{t_1}\frac{dtdt_1\ldots dt_{j-1}}{\nu(t)} \\
	&\quad + \int_{A|\alpha|}^{\sqrt{A_p^2+A^2}}\int_0^{t_{j-1}}\cdots\int_0^{t_1}\frac{dtdt_1\ldots dt_{j-1}}{\nu(t)} \\
	&= C_{p,j} + I^j_\nu(z'_p),\endaligned \]
	where $C'_{p,j}  = \int_{A|\alpha|}^{\sqrt{A_p^2+A^2}}\int_0^{t_{j-1}}\cdots\int_0^{t_1}\frac{dtdt_1\ldots dt_{j-1}}{\nu(t)} < \infty.$
		Obviously, in the case $|z'_p|\le |\alpha|,$ we have 
	\[I^j_\nu(\gamma_p(z))\le  \int_0^{A|\alpha|}\int_0^{t_{j-1}}\cdots\int_0^{t_1}\frac{dtdt_1\ldots dt_{j-1}}{\nu(t)} <\infty.\]
	
	At that point, using reasoning similar to the final estimates in the proof of assertion (a), we obtain assertion \eqref{lem_gamma*1}.
	 
	 (c) 
	 For every  $r \in (|\alpha|, 1),$ the continuity of $\gamma$ ensures that the set   $\{\gamma(z): |z|\le r\}$ is   compact in $\mathbb B.$ Since $\nu$ is positive and continuous, $\inf_{|z|\le r}\nu(\gamma(z)) > 0,$ it implies that
	\[\sup_{|z|\le r}\frac{\nu(z)}{\nu(\gamma(z))} < \infty.\]
		On the other hand, for every $z \in \mathbb B,$ $|z|> r,$ by \eqref{est_gamma_k}, we have
	\[\aligned
	\frac{(1-|z|)^a}{(1-|\gamma_\alpha(z)|)^b} &\le \frac{(1-r)^a}{\Big(1 - \frac{2r^2}{|\alpha|(1+|\alpha|)} - \frac{1-r}{1 +r}|z|^2\Big)^b} \to \frac{(1-r)^a}{\Big(1 - \frac{2r^2}{|\alpha|(1+|\alpha|)} - \frac{1-r}{1 +r}\Big)^b} < \infty
	\endaligned\]
	as $|z|\to 1$ because it is easy to checck that $1 - \frac{2r^2}{|\alpha|(1+|\alpha|)} - \frac{1-r}{1 +r} > 0.$ Therefore,	\[\sup_{|z|> r}\frac{(1-|z|)^a}{(1-|\gamma_\alpha(z)|)^b} < \infty\]
		  for $r \in (|\alpha|,1)$ sufficiently large. Then,	by \eqref{w_1} and \eqref{w_2},	 
		 \begin{equation}\label{eq_prepare4}\lim_{|z|\to 1}\frac{\nu(z)}{\nu(\gamma(z))} = \lim_{|z|\to 1}\frac{\nu(z)}{(1 - |z|)^a}\frac{(1 - |\gamma(z))|)^b}{\nu(\gamma(z))}\frac{(1 - |z|)^a}{(1-|\gamma(z)|)^b} = 0. \end{equation} 
	Then, we obtain  the first inequality. 
	
	Now, it is obvious that $\sup_{|\gamma_k(z)|\le r}\frac{\nu(z_k)}{\nu(\gamma_k(z)} <\infty$ with $r\in (\delta, 1).$ In the case $|\gamma_k(z)|> r,$ since $|\gamma_k(z)|\le |\gamma(z)|,$ by an estimate as \eqref{eq_prepare4}, we obtain the second inequality in (b) of the lemma.
\end{proof}

\begin{remark}\label{rmk_h} Since $\M^j_p < \infty$ we can find $C_j^*>0$ such that 
\[
\sup_{w\in \BB}h(w)\int_0^{\delta_0}\int_0^{t_{j-1}}\cdots\int_0^{t_1}\frac{dtdt_1\ldots dt_{j-1}}{m_{\nu, \delta_0}} \le C^*_j\sup_{|\varphi_p(w)|>\delta_0}h(w)I^j_\nu(\varphi_p(w)).\]
Thus, the estimate \eqref{lem_S*1} can be written as follows: 
\[
\sup_{|\varphi(z)|>\delta_0}h(z)\Big\|\delta^{\mathcal H^{(j)}_\nu}_{\varphi(z)}\Big\| \lesssim \sup_{|\varphi_p(w)|>\delta_0}h(w)\Big\|\delta^{\mathcal H^{(j)}_\nu}_{\varphi_p(w)}\Big\|,\]
\end{remark}

\begin{lemma}\label{lem_prepare_3} Let $ \nu $ be  a  normal  weight   on $ \BB, $  $\alpha \in \BB \setminus \{0\}$ and $\gamma \in Aut(\BB)$   defined by
	\eqref{eq_Mobius}. 
	Then, the composition operator $C_\gamma: \mathcal{H}^{(n)}_{\omega} \to \mathcal{H}^{(n)}_{\omega},$ $f \mapsto f \circ \gamma,$ is an homeomorphism. 
\end{lemma} 
\begin{proof} Note that $ \gamma_j \in H(\overline{\BB}), $ $ j=1, \ldots,N, $ it implies from \eqref{eq_Mobius} and Corollary 1.5 in \cite{Zh}
 that $ R^{(k)}\gamma_j  \in H(\overline{\BB})$  and $ R^{(k)}\gamma_j$ is bounded in $ \overline{\BB} $ for any
positive integer $ k, $ i.e.,
\begin{equation}\label{estMRphi} M_{\gamma}^{(k)} := \sup_{z\in\BB}|R^{(k)}\gamma(z)|<\infty, \quad k = 1, 2, \ldots \end{equation}

By   \eqref{eq_Bruno}, 
\eqref{compair_norm},  Lemma \ref{lem_B_n}(c), 
we obtain
\[\aligned
\|C_\gamma(f)\|_{\mathcal H^{(n)}_\omega} &= \sup_{z\in \BB}\omega(z)|R^{(n)}(f\circ\gamma)(z)| \\
&\le \sup_{z\in \BB}\frac{\omega(z)}{\omega(\gamma(z))}\sum_{j=1}^n \sum_{\vec{l} \in L_{j}} \omega(\gamma(z))\bigg|\frac{\partial^j f(\gamma(z))}{\partial z_{\vec{l}}}\bigg|
\sum_{\vec{k} \in K_{n,j}} C_{\vec{k}}^{n} \prod_{t=1}^j |R^{(k_t)} \gamma_{l_t}(z)\Big|\\
&\lesssim \sup_{z\in \BB}\frac{\omega(z)}{\omega(\gamma(z))}\bigg(\sum_{j=1}^n \sum_{\vec{k} \in K_{n,j}} C_{\vec{k}}^{n} \prod_{t=1}^j M^{(k_t)}_\gamma\bigg) \|f\|_{\mathcal H^{(j)}_\omega}
\\
&\lesssim \sup_{z\in \BB}\frac{\omega(z)}{\omega(\gamma(z))}\bigg(\sum_{j=1}^n \sum_{\vec{k} \in K_{n,j}} C_{\vec{k}}^{n} \prod_{t=1}^j M^{(k_t)}_\gamma\bigg) \|f\|_{\mathcal H^{(n)}_\omega}.
\endaligned\] 
This means $C_\gamma$ is bounded. Since $\gamma \in Aut(\BB)$ it is easily seen that $C_{\gamma^{-1}} = C^{-1}_\gamma$ is also bounded. Hence, the lemma is proved. 
\end{proof}

\section{The Condition on the Symbols}\label{sec_4}
In this section, let $ \psi \in H(\BB), $ $\varphi = (\varphi_1, \ldots, \varphi_N) \in S(\BB),$ and   $ \mu, \nu $ be  normal weights on $ \BB. $ 
We use there certain quantities, which will be used in the main results of this paper:
 \[\aligned
 \mathfrak B_{n,j}(\varphi(z)) &:=  \sum_{\vec{k} \in K_{n,j}} \sum_{\vec{l} \in L_{j}} C_{\vec{k}}^{n} \prod_{t=1}^j R^{(k_t)} \varphi_{l_t}(z), \\
 \mathfrak B_{n,j}(\varphi_p(z)) &:=  \sum_{\vec{k} \in K_{n,j}}  C_{\vec{k}}^{n}\prod_{t=1}^jR^{(k_t)} \varphi_{p}(z), \\
\mathscr B_{0}^n(\psi; \varphi_p)(z) &:= \mathscr B_{0}^n(\psi; \varphi)(z)  := R^{(n)}(\psi(z)), \\
\mathscr B_{j}^n(\psi; \varphi_*)(z) &:=   \sum_{i=j}^n\binom{n}{i}R^{(n-i)}(\psi(z)) \mathfrak B_{i,j}(\varphi_*(z))  \quad\text{for}\ j\ge 1.\\
\endaligned\]
Here, the notation $\varphi_*$ denotes either $\varphi$ or $\varphi_k,$ $k = 1, \ldots, N.$

By performing similar calculations as in this formula with $\vec{l} = (p, \ldots, p) \in L_{j_0}$ we obtain
    \begin{equation}\label{formula_psi_phi} 
   R^{(n)}\Big(\psi\cdot \varphi^{j_0}_p\Big)(z) 
= \sum_{i=0}^{j_0}\mathscr B^n_{j_0-i}(\psi; \varphi_p)(z)\varphi^i_p(z).\end{equation}
 The following estimate is written from the formula (4.3) in \cite{Qu1}:
     \begin{equation}\label{est_bounded_1}
 \mu(z)|R^{(n)}W_{\psi,\varphi}(f)| \lesssim \sum_{j=0}^n  \mu(z)\big|\mathscr B_{j}^n(\psi; \varphi)(z)\big|\Big\|\delta_{\varphi(z)}^{\mathcal{H}^{(n+m-j)}_\nu}\Big\| \|f\|_{\mathcal{H}^{(n+m)}_\nu}\end{equation}
    
    In the assumptions of the main theorems of this paper, we use   the following condition:      The pair of functions $(\psi, \varphi_p)$ is said to satisfy the $(n, \mu)$-condition if
    \[\aligned
    \psi \in \mathcal H^{(n)}_{\mu,+} &:= \Big\{f \in \mathcal{H}^{(n)}_{\nu}: \ \lim_{|z|\to 1}\nu(z)|R^{(n)}f(z)| > 0\Big\},\\
    \psi\cdot\varphi^j_p \in \mathcal H^{(n)}_{\mu,0}&:= \Big\{f \in \mathcal{H}^{(n)}_{\nu}:   \lim_{|z| \to1}\nu(z)|R^{(n)}f(z)| =0\Big\} ,\quad j = 1.\ldots, n.
   \endaligned\]
     Below, we will present some examples to demonstrate that  assumption ($n, \mu$) is valid.
     
      \begin{example} For $\alpha \in (0, 1),$  consider the weight $\mu(z) \in H(\BB),$   $\varphi \in S^*_p(\BB)$ with
    \[\mu(z) = (1-|z|^2)^\alpha,   \quad   
 \varphi_p(z)= \langle z, e_p\rangle.  
\]

    We construct the function $\psi \in H(\BB)$  as follows:
  Let $q>1$ be a large positive integer to be determined, define
\begin{equation}\label{exam_1_Psi}
\widetilde\psi(z)=\sum_{k=0}^{\infty} a_k \langle z,e_p\rangle^{n_k} = \sum_{k=0}^{\infty} a_kz_p^{n_k},
\end{equation}
where $a_k=q^{k(\alpha-1)+\frac{\alpha}{2}}, n_k=q^k$. Because $\widetilde\psi(z)$ is a lacunary power series with
$$
a_k n_k^{1-\alpha}=q^{k(\alpha-1)+\frac{\alpha}{2}} q^{k(1-\alpha)}=q^{\frac{\alpha}{2}},
$$
using Theorem 1 (1) in \cite{Ya} we have $\widetilde\psi \in \mathcal B^\alpha = \mathcal H^{(1)}_\mu,$ and since $\alpha \in (0,1),$ it is easy to check that $\widetilde\psi \in  H^\infty(\BB).$  

By modifying an argument in the proof of Theorem 6 in \cite{Gi}, we next will show 
\begin{equation}\label{gap_0}
 |\nabla\widetilde\psi(z)(z)| \gtrsim \frac{1}{(1-|z|)^\alpha}
\end{equation}
for all $z \in \BB$ sufficiently close to the boundary. 

We write
$$
\begin{aligned}
|R\widetilde\psi(z)(z)| & =\bigg|\sum_{i=0}^{\infty} q^{i(\alpha-1)+\frac{\alpha}{2}+i} z_p^{q^i}\bigg| \\
& \ge q^{k(\alpha-1)+\frac{\alpha}{2}+k}|z_p|^{q^k+1}-\sum_{i=0}^{k-1} q^{i(\alpha-1)+\frac{\alpha}{2}+i}|z_p|^{q^i}-\sum_{i=k+1}^{\infty} q^{i(\alpha-1)+\frac{\alpha}{2}+i}|z_p|^{q^i} \\
& \ge q^{k(\alpha-1)+\frac{\alpha}{2}+k}|z_p|^{q^k+1}-\sum_{i=0}^{k-1} q^{i(\alpha-1)+\frac{\alpha}{2}+i}|z|^{q^i}-\sum_{i=k+1}^{\infty} q^{i(\alpha-1)+\frac{\alpha}{2}+i}|z|^{q^i} \\
& :=Q_1-Q_2-Q_3.
\end{aligned}
$$

For $z$ satisfying
\begin{equation}\label{gap_1}
1-\frac{1}{q^k} \le |z_p|\le |z| \le 1-\frac{1}{q^{k+\frac12}}\end{equation}
we have
\begin{equation}\label{gap_2}
|z_p|^{q^k+1} \geq \Big(1-\frac{1}{q^k}\Big)^{q^k+1} \ge \frac13
\end{equation}
if $q$ is large enough. Then \eqref{gap_2} gives
\[\aligned
Q_1 &\ge \frac13 q^{\left(k+\frac{1}{2}\right) \alpha}, \\
Q_2 &\le \sum_{i=0}^{k-1} q^{i(\alpha-1)+\frac{\alpha}{2}+i}   =\sum_{i=0}^{k-1} q^{i \alpha+\frac{\alpha}{2}}   =q^{\frac{\alpha}{2}} \frac{q^{k \alpha}}{q^\alpha-1}   =\frac{q^{\left(k+\frac{1}{2}\right) \alpha}}{q^\alpha-1}.\endaligned
\]

Applying \eqref{gap_1} again, we have $|z_p|^{q^k} \leq\left(\frac{1}{2}\right)^{q^{-\frac{1}{2}}}$ and
\[\aligned
Q_3 & =\sum_{i=k+1}^{\infty} q^{i \alpha+\frac{\alpha}{2}}|z|^{q^i}  =q^{\frac{\alpha}{2}} \sum_{i=k+1}^{\infty} q^{i \alpha}|z|^{q^i} \\
& \leq q^{\frac{\alpha}{2}} q^{\alpha(k+1)}|z|^{q^{k+1}} \sum_{i=0}^{\infty}\left(q^\alpha|z|^{q^{k+2}-q^{k+1}}\right)^i \\
& =q^{\alpha(k+1)+\frac{\alpha}{2}} \frac{|z|^{k+1}}{1-q^\alpha|z|^{q^{k+2}-q^{k+1}}}   =q^{\alpha\left(k+\frac{1}{2}\right)} \frac{q^\alpha|z|^{k+1}}{1-q^\alpha|z|^{q^{k+2}-q^{k+1}}} \\
& \leq q^{\left(k+\frac{1}{2}\right) \alpha} \frac{q^\alpha\left(\frac{1}{2}\right)^{\frac{1}{2}}}{1-q^\alpha\left(\frac{1}{2}\right)^{\frac{3}{2}-q^{\frac{1}{2}}}}.
\endaligned\]
From \eqref{gap_1} we have $
q^{k+\frac{1}{2}} \geq \frac{1}{1-|z|}.$ 
Combining \eqref{gap_2} with the estimates for $Q_1$, $Q_2$, and $Q_3,$  we get
\[
 |R\widetilde\psi(z)|  \geq \frac14 q^{\left(k+\frac12\right) \alpha}   \geq \frac14 \frac{1}{(1-|z|)^\alpha}
\]
for $z$ satisfying \eqref{gap_1} and $q$ sufficiently large, hence, \eqref{gap_0} is proved. This implies that
\[\lim_{|z|\to 1}(1-|z|^2)^\alpha|R\widetilde\psi(z)(z)| > 0.\]

Now we put
\[\psi(z) := \int_0^{\langle z, e_p\rangle}\int_0^{w_p^{(n-2)}}\ldots\int_0^{w_p^{(1)}}\widetilde\psi(z)(t)dtdt_p^{(1)}\cdots dt_p^{(n-2)}, \quad z \in \BB.\]
It is easy to verify that $R^{(n)}\psi(z) =  R\widetilde\psi(z)(z)z_p^{n-1}.$ Consequently,
\[\lim_{|z|\to 1}(1-|z|^2)^\alpha|R^{(n)}\psi(z)| \gtrsim  \lim_{|z|\to 1}(1-|z|^2)^\alpha|R\widetilde\psi(z)(z)| > 0.\]
That means $\psi \in \mathcal H^{(n)}_{\mu,+}.$

On the other hand, since $\widetilde\psi \in H^\infty(\BB),$ there exists  $M>0$ such that
 $\sup_{z\in \BB}|R\psi(z)| \le M,$  hence, $\sup_{z\in \BB}|R^{(m)}\psi(z)|<M$  for every $m = 0, 1, \ldots, n.$ Then,
 \[\aligned
 \lim_{|z|\to 1}&(1-|z|^2)^\alpha|R^{(n)}(\psi\cdot\varphi^j_p)(z)|\\
 & \le \sum_{i=0}^n\binom{n}{i} \lim_{|z|\to 1}(1-|z|^2)^\alpha|R^{(n-i)}\psi(z)R^{(i)}\varphi^{j_0}_p(z)|\\
 &\le \sum_{i=0}^n\binom{n}{i} \lim_{|z|\to 1}(1-|z|^2)^\alpha M\frac{j_0!}{(j_0-i)!}|z_p^{j-i}|=0\endaligned\]
 for every $j=1, \ldots, n.$ Therefore, $\psi\cdot\varphi_p^j \in \mathcal H^{(n)}_{\mu,0}$  for every $j=1, \ldots, n.$
  \end{example}

    \begin{example}  Consider the weight $\mu(z) \in H(\BB),$   $\varphi \in S^*_p(\BB)$ with
    \[\mu(z) = (1-|z|^2)^\alpha,   \quad   
 \varphi_p(z)= \frac{z_p -a_p}{1 - z_p\overline{a}_p} , 
\]
    where $\alpha \in (0, 1), $ $a \in \BB$ and $\psi = \varPsi$ which is defined by \eqref{exam_1_Psi}. 
   
   First, we check that $\psi \in \mathcal H^{(n)}_{\mu,+}.$
   
   It is clear that
   \[R^{(n)}\psi(z) = \sum_{i=0}^{\infty} q^{i(\alpha-1)+\frac{\alpha}{2} +ni}z_p^{q^i}.\]
   Then, using a similar calculation as in the above example, we have
   \[\aligned
   |R^{(n)}\psi(z)| & =\bigg|\sum_{i=0}^{\infty} q^{i(\alpha-1)+\frac{\alpha}{2}+ni} z_p^{q^i}\bigg| \\
   & \ge q^{k(\alpha-1)+\frac{\alpha}{2}+nk}|z_p|^{q^k+1}-\sum_{i=0}^{k-1} q^{i(\alpha-1)+\frac{\alpha}{2}+ni}|z_p|^{q^i}-\sum_{i=k+1}^{\infty} q^{i(\alpha-1)+\frac{\alpha}{2}+ni}|z_p|^{q^i} \\
      & :=Q'_1-Q'_2-Q'_3.
     \endaligned\]
     For $z$ satisfying 
     \begin{equation}\label{gap_1a}
1-\frac{1}{q^k} \le |z_p|\le |z| \le 1-\frac{1}{q^{k+\frac32}},\end{equation} as in the above, we also   have \eqref{gap_2} and 
     \[\aligned
     Q'_1 &\ge \frac13 q^{(k+1)(n-1)+\left(k+\frac{3}{2}\right) \alpha}, \\
     Q'_2 &\le  \frac{q^{(k+1)(n-1)+\left(k+\frac{3}{2}\right) \alpha}}{q^{n+\alpha-1}-1},\\
     Q'_3 &\le   q^{(k+1)(n-1)+\left(k+\frac{3}{2}\right) \alpha} \frac{q^{n+\alpha-1}\left(\frac{1}{2}\right)^{\frac{1}{2}}}{1-q^{n+\alpha-1}\left(\frac{1}{2}\right)^{\frac{3}{2}-q^{\frac{1}{2}}}}.
     \endaligned\]
     From \eqref{gap_1a} we have $
q^{k+\frac{3}{2}} \geq \frac{1}{1-|z|}.$ 
Combining \eqref{gap_2} with the estimates for $Q'_1$, $Q'_2$, and $Q'_3,$  we get
\[
 |R^{(n)}\psi(z)|  \geq \frac14 q^{(k+1)(n-1)+\left(k+\frac32\right) \alpha}   \geq \frac14 \frac{1}{(1-|z|)^\alpha}
\]
for $z$ satisfying \eqref{gap_1a} and $q$ sufficiently large. This implies that
\begin{equation}\label{gap_3}
\lim_{|z|\to 1}(1-|z|^2)^\alpha|R^{(n)}\psi(z)| > 0.
\end{equation} 
Thus, $\psi \in \mathcal H^{(n)}_{\mu,+}.$
   
   Finally, by similar arguments and estimates as in Example 1, we can also easily prove that $\psi\cdot\varphi_p^j \in \mathcal H^{(n)}_{\mu,0}$  for every $j=1, \ldots, n.$
    \end{example}
   
   \section{Boundedness and Compactness of the Operator $W_{\psi,\varphi}$}\label{sec_5}
 In this section we will characterize the boundedness and the compactness of weighted composition operator $W_{\psi,\varphi}: \mathcal H^{(k)}_\nu \to \mathcal H^{(n)}_\mu$ in  both cases $k\ge n$ and $k < n.$ 
 
 We  need the following lemmas to prepare for proving the main theorems of the paper.
 \begin{lemma}\label{lem_A1} Assume that 
 $\varphi(0)=0$ and $\psi, \varphi_p$ satisfy the condition ($n, \mu$). Then, there exists $\lambda \in (0,1)$ such that
	\begin{equation}\label{eq_inf_A}
		\inf_{|\varphi_p| >\lambda}\mathscr B^{n-}_{j,p} := \inf_{|\varphi_p(z)|>\lambda}\mu(z)\big|\mathscr B_{j}^n(\psi; \varphi_p)(z)\big|   >0\quad\text{for every $j= 0, 1, \ldots, n.$}\end{equation}
		 \end{lemma}
 \begin{proof} It follows from the hypothesis $\psi \in \mathcal H^{(n)}_{\mu,+}$ and $|\varphi_p(z)|\le |z|$ that 
		 \[\lim_{|\varphi_p(z)|\to 1}\mu(z)|\mathscr B_{0}^n(\psi; \varphi_p)(z)| = \lim_{|z|\to 1}\mu(z)|\mathscr B_{0}^n(\psi; \varphi_p)(z)|= \mu(z)|R^{(n)}\psi(z)| > 0.\] 
		 Thus, \eqref{eq_inf_A} holds for $j=0.$
		 
		 Denote
		 \[K^{0,r_1,\ldots,r_s}_{i,j}=\left\{\ \begin{array}{l|l}\hskip-0.2cm\vec{k}  \in  K^0_{i,j} & \begin{array}{l} k_{r_1}=\ldots = k_{r_s} =0, \\ k_t \neq 0 \ \text{if $t\neq r_1, \ldots, r_s$}\end{array}\end{array}\hskip-0.3cm\ \right\}, \quad s = 1, \ldots, j.\]
		 
		 For any $j_0 \in \{1, \ldots, n\},$  since   a vector $\vec{k}\in K^{0, r_1, \ldots, r_s}_{i, j_0}\setminus K_{i, j_0}\subset K^{0}_{i, j_0}\setminus K_{i, j_0}$
  can be considered as $\vec{k}\in K_{i, s},$  and conversely, each vector $\vec{k}\in K_{i, s},$ there exist $j_0$ 
  vectors in $\vec{k}\in K^{0}_{i, j_0}\setminus K_{i, j_0}$
 that can be identified with it in the aforementioned sense, we have
	  \[\aligned
	  \mu(z)&R^{(n)}\big(\psi\cdot\varphi^{j_0}_p\big)(z) = \mu(z)\sum_{i=0}^n\binom{n}{i}R^{(n-i)}\psi(z)R^{(i)}\varphi^{j_0}_p(z) \\
	  &=\mu(z)R^{(n)}\psi(z)\varphi^{j_0}_p(z) + \mu(z)\sum_{i=1}^n\binom{n}{i}R^{(n-i)}\psi(z)R^{(i)}\varphi^{j_0}_p(z)\\
	 &=\mu(z)R^{(n)}\psi(z)\varphi^{j_0}_p(z) +  \mu(z)\sum_{i=1}^n\binom{n}{i}R^{(n-i)}\psi(z)\sum_{\vec{k}\in K^0_{i,j_0}\setminus K_{i,j_0}}C^i_{\vec{k}}\prod_{t=1}^{j_0}R^{(k_t)}\varphi_p(z)\\  
	 &\quad +  \mu(z)\sum_{i=j_0}^n\binom{n}{i}R^{(n-i)}\psi(z)\sum_{\vec{k}\in K_{i,j_0}}C^i_{\vec{k}}\prod_{t=1}^{j_0}R^{(k_t)}\varphi_p(z) \\
	 &=\mu(z)R^{(n)}\psi(z)\varphi^{j_0}_p(z) \\
	 &\quad+   j_0 \mu(z)\sum_{i=1}^n\binom{n}{i}R^{(n-i)}\psi(z) \sum_{\vec{k}\in K_{i,1}}C^i_{\vec{k}}R^{(i)}\varphi_p(z)\varphi^{j_0-1}_p(z) 	 \\
	 &\quad+    j_0\mu(z)\sum_{i=1}^n\binom{n}{i}R^{(n-i)}\psi(z)\sum_{\vec{k}\in K_{i,2}}C^i_{\vec{k}}\prod_{t=1}^{2}R^{(k_t)}\varphi_p(z)\varphi^{j_0-2}_p(z) 	 \\	 	
	   &\quad+ \cdots + \cdots \\
	 &\quad+  j_0\mu(z)\sum_{i=1}^n\binom{n}{i}R^{(n-i)}\psi(z)\sum_{\vec{k}\in K_{i,j_0-2}}C^i_{\vec{k}}\prod_{t=1}^{j_0-2}R^{(k_t)}\varphi_p(z)\varphi_p(z) 	 \\
	  &\quad+  j_0\mu(z)\sum_{i=1}^n\binom{n}{i}R^{(n-i)}\psi(z)\sum_{\vec{k}\in K_{i,j_0-1}}C^i_{\vec{k}}\prod_{t=1}^{j_0-1}R^{(k_t)}\varphi_p(z)\varphi^2_p(z) 	 \\
	 &\quad +\mu(z)\sum_{i=1}^n\binom{n}{i}R^{(n-i)}\psi(z)\sum_{\vec{k}\in K_{i,j_0}}C^i_{\vec{k}}\prod_{t=1}^{j_0}R^{(k_t)}\varphi_p(z)\\
	 &=j_0\mu(z)R^{(n)}\big(\psi\cdot\varphi_p\big)(z)\varphi^{j_0-1}_p(z) - (j_0-1)R^{(n)}\psi(z)\varphi^{j_0}_p(z) \\
	 &\quad+j_0\mu(z)R^{(n)}\big(\psi\cdot\varphi^2_p\big)(z)\varphi^{j_0-2}_p(z) - j_0R^{(n)}\psi(z)\varphi^{j_0-1}_p(z)\\	 
	 &\quad+ \cdots  \\
	 	 &\quad+j_0\mu(z)R^{(n)}\big(\psi\cdot\varphi^{j_0-2}_p\big)(z)\varphi^{2}_p(z) - j_0R^{(n)}\psi(z)\varphi^{3}_p(z)\\	 
		 &\quad+j_0\mu(z)R^{(n)}\big(\psi\cdot\varphi^{j_0-1}_p\big)(z)\varphi_p(z) - j_0R^{(n)}\psi(z)\varphi^{2}_p(z)\\
	 &\quad +\mathscr B^n_{j_0}(\psi; \varphi_p)(z).
	 \endaligned
	  \]
	  As in the above, by $\psi \in \mathcal H^{(n)}_{\mu, +}, $	$\psi\cdot\varphi^j_p \in \mathcal H^{(n)}_{\mu, 0} $ for every $j=1, \ldots j_0,$ 	this implies that
	\[\lim_{|\varphi_p(z)|\to 1}\mu(z)|\mathscr B_{j_0}^n(\psi; \varphi_p)(z)| >0.\]
	We have the lemma to be proved.
 \end{proof}
 
 By using reasoning similar to that in the proof of Lemma 5.1 in \cite{Qu1} for the function $\frac{\partial}{\partial z_{l_1}}\Big(\frac{\partial^{j-1}f_s}{\partial z_{l_2}\cdots\partial z_{l_j}}\Big),$ $\vec{l} \in L_j,$ we obtain a similar result and will omit its proof.

 \begin{lemma} \label{lem_I_infty}  Assume $\nu$ is a normal weight on $\BB$ and 
  \[I^{m-i}_\nu(1) = \int_0^1\int_0^{t_{m-i-1}}\cdots\int_0^{t_1}\frac{1}{\nu(t)}dtdt_1\cdots dt_{m-i-1} <\infty\] 
 holds for some $i \in \{0, 1, \ldots,m\}.$  Then, for every bounded sequence $\{f_s\}_{s \ge 1} \subset \mathcal{H}^{(m)}_\nu$ converging to 0 uniformly on compact subsets of $\BB$, we have 
\[\lim _{s \to \infty} \sup _{z \in \BB}\bigg|\frac{\partial^j f_s(\varphi(z)) }{\partial z_{\vec{l}}}
 \bigg| =0 \quad\text{for}\    j = 0, \ldots, i.\]
\end{lemma}

Now we characterize the boundedness   of weighted composition operator $W_{\psi,\varphi}$.

 \begin{theorem}\label{thm_A1} Let $n, m \in\N_0.$  Assume that $\varphi  \in \widetilde S_p(\BB)$ for some $p\in \{{1,\ldots,N\}},$ such that the condition ($n, \mu$) satisfied.
  The  following are equivalent:  
\begin{enumerate}
	\item[\rm1)]  $ W_{\psi,\varphi}:  \mathcal{H}^{(n+m)}_{\nu} \to \mathcal{H}^{(n)}_{\mu}$ is bounded;
	\item[\rm2)] $\psi, \psi\cdot\varphi^i_p  \in \mathcal{H}^{(n)}_{\mu}$ for every  $i = 0,1,2 \ldots, $ and
		\begin{equation}\label{est_B_1} 
  \mathscr B_{j,p}^n := \sup_{z\in \BB}\mu(z)\big|\mathscr B^n_{j}(\psi; \varphi_p)(z)\big|\Big\|\delta_{\varphi_p(z)}^{\mathcal{H}^{(n+m-j)}_{\nu}}\Big\|  <\infty \quad \text{for every}\ 0\le j \le  n.\end{equation}
\end{enumerate}
In this case, 
\begin{equation}\label{norm_A}
\|W_{\psi,\varphi}\| \asymp |\psi(0)|\Big\|\delta_{\varphi(0)}^{\mathcal H^{(n+m)}_{\nu}}\Big\| + \sum_{j=0}^n \mathscr B_{j,p}^n. \end{equation}
 \end{theorem}
\begin{proof} First,   using the same argument as in the proof of Theorem 4.1 in \cite{Qu1}, we obtain \eqref{norm_A} in the case where $ W_{\psi,\varphi}:  \mathcal{H}^{(n+m)}_{\nu} \to \mathcal{H}^{(n)}_{\mu}$ is bounded.

1) $\Rightarrow$ 2): 
It follows from Theorem 4.1 in \cite{Qu1} that $\psi, \psi\cdot\varphi^i_p  \in \mathcal{H}^{(n)}_{\mu}$ for every  $i = 0,1,2 \ldots, $ and 
  \[\mathscr B_{j}^n := \sup_{z\in \BB}\mu(z)\big|\mathscr B^n_{j}(\psi; \varphi)(z)\big|\Big\|\delta_{\varphi(z)}^{\mathcal{H}^{(n+m-j)}_{\nu}}\Big\|  <\infty,\] 
  hence, \eqref{est_B_1} holds.
  
 2) $\Rightarrow$ 1): By the hypothesis $\psi, \psi\cdot\varphi^i_p  \in \mathcal{H}^{(n)}_{\mu}$ for every  $i = 0,1,2 \ldots, $ by induction on $j,$  a proof step of Theorem 4.1 in [Qu] has shown that
\begin{equation}\label{eq_B+-}  \mathscr B_{j}^{n-} := \sup_{z\in\BB}\mu(z)\big|\mathscr B_{j}^n(\psi; \varphi)(z)\big| <\infty,\end{equation}
 for any $j=0, 1, \ldots, n.$  
 
 Since $\varphi \in \widetilde S_p(\BB),$ there exists $\alpha \in \BB$ such that $\varphi(\alpha) = 0.$ 
 
 $\bullet$ First, we consider the case $\alpha=0,$ i.e., $\varphi(0) = 0.$
 
By Lemma \ref{lem_A1}, there exists $\lambda \in (0,1)$ such that $\inf_{|\varphi_p|>\lambda}\mathscr B^{n-}_{j,p} >0. $

Combinging with \eqref{eq_B+-},  we have
	\[D_j := \frac{\mathscr B_{j}^{n-}}{\inf_{|\varphi_p|>\lambda}\mathscr B^{n-}_{j,p}} <\infty.\]
 
 Then, by $\varphi \in \widetilde S_p(\BB),$ for each $z \in \BB,$ $|\varphi(z)|>\lambda$ (hence, $|z|>\lambda$)  there exists $z' \in \BB,$ such that $|\varphi(z)| = |\varphi_p(z')|$ (hence, $|z'|>\lambda$). Therefore,
  by appliying Lemma \ref{lem_B_n} to the functions $h_j(z) := \mu(z)\mathscr B_{j}^n(\psi; \varphi_p)(z),$ from the estimate \eqref{est_bounded_1} we have
 \begin{equation}\label{es_final}\aligned
 \mu(z)|R^{(n)}&W_{\psi,\varphi}(f)(z)| \lesssim \sum_{j=0}^n D_j  \mu(z')\big|\mathscr B_{j}^n(\psi; \varphi_p)(z')\big|\Big\|\delta_{\varphi(z)}^{\mathcal{H}^{(n+m-j)}_\nu}\Big\| \|f\|_{\mathcal{H}^{(n+m)}_\nu} \\
 &\lesssim \sum_{j=0}^n D_j  \sup_{w\in \BB}\mu(w)\big|\mathscr B_{j}^n(\psi; \varphi_p)(w)\big|\Big\|\delta_{\varphi_p(w)}^{\mathcal{H}^{(n+m-j)}_\nu}\Big\| \|f\|_{\mathcal{H}^{(n+m)}_\nu}\\
 &= \bigg(\sum_{j=0}^n D_j   \mathscr B_{j,p}^n \bigg)\|f\|_{\mathcal{H}^{(n+m)}_\nu}.
 \endaligned\end{equation}
 Consequently, 
	\[
		\|W_{\psi,\varphi}f(z)\|_{s\mathcal H^{(n))}_\mu} \lesssim \bigg(\sum_{j=0}^n  \mathscr B_{j,p}^n \bigg)\|f\|_{\mathcal{H}^{(n+m)}_\nu}.
	 \]
	 This implies that $W_{\psi, \varphi}$ is bounded.
	 
	 \medskip
	$\bullet$ Next, we consider the case   $\alpha \neq 0,$  i.e.,  $\varphi(\alpha) = 0.$   
	
       Let $\gamma_\alpha \in Aut(\mathbb B)$ given by \eqref{eq_Mobius}. Then $\eta := \varphi \circ \gamma$ satisfies $\eta(0) = 0$ because $\varphi(\alpha) = 0.$ Since $\gamma$ is an automorphism, it is obvious that $\eta \in \widetilde S_p(\BB).$ It is clear that $\eta_p(z) = \varphi_p(\gamma_\alpha(z)).$
       
       The proof of the boundedness of $W_{\psi, \eta}$ will be completed by applying the case $\alpha = 0$ above after verifying that $\psi,$ $\eta_p$ satisfy the condition ($n, \omega$) and
 \begin{equation}\label{bounded_1_eta_p} 
  \mathscr B_{j,\eta_p}^n := \sup_{z\in \BB}\mu(z)\big|\mathscr B^n_{j}(\psi; \eta_p)(z)\big|\Big\|\delta_{\eta_p(z)}^{\mathcal{H}^{(n+m-j)}_{\nu}}\Big\|  <\infty\quad\text{for every $0\le j \le  n.$}
  \end{equation}

It follows from \eqref{eq_modul_phi} that
\[|z|\to 1\quad \Leftrightarrow \quad |\gamma_\alpha(z)| \to 1.\]
This implies that, for every $j=1, \ldots, n,$
    \begin{equation}\label{B_Reta_Rphi}
  \aligned
  \lim_{|z|\to 1} \mu(z)&|\mathscr B^n_{j}(\psi; \eta_p)(z)| \\
  &=  \lim_{|\gamma_\alpha(z)|\to 1} \Bigg|\sum_{i=j}^n\binom{n}{i}\mu(\gamma_\alpha(z))R^{(n-i)}(\psi(\gamma_\alpha(z))) \mathfrak B_{i,j}(R\varphi_p(\gamma_\alpha(z))\Bigg| \\
   &= \lim_{|z'|\to 1}\mu(z')|\mathscr B^n_{j}(\psi; \varphi_p)(z')|, \\
   \lim_{|z|\to1}\mu(z)&|R^{(j)}\eta_p(z)| =\lim_{|\gamma_\alpha(z)|\to1}\mu(z)|R^{(j)}\varphi_p(\gamma_\alpha(z))| \\
   & =  \lim_{|z'|\to1}\mu(z')|R^{(j)}\varphi_p(z')|. 
    \endaligned
    \end{equation}
      
      Since $ \psi\cdot\varphi^j_p \in \mathcal H^{(n)}_{\mu,0},$ by \eqref{formula_psi_phi} and  \eqref{B_Reta_Rphi}, 
 we have
\[\aligned 
\lim_{|z|\to1}\mu(z)\Big|R^{(n)}\Big(\psi\cdot \eta^{j}_p(z)\Big)\Big|
&=\lim_{|z|\to1}\bigg|\sum_{i=0}^{j_0}\mu(z)\mathscr B^n_{j_0-i}(\psi; \eta_p)(z)\eta^i_p(z)\bigg| \\
&= \lim_{|z'|\to1}\bigg|\sum_{i=0}^{j_0}\mu(z')\mathscr B^n_{j_0-i}(\psi; \varphi_p)(z')\varphi^i_p(z')\bigg|\\
&=\lim_{|z'|\to1}\mu(z')\Big|R^{(n)}\Big(\psi\cdot \varphi^{j}_p(z')\Big)\Big| = 0.\\
\endaligned\] 
Thus, $ \psi\cdot\eta^j_p \in \mathcal H^{(n)}_{\mu,0}$ for every $j = 1, \ldots, n.$ 

   \medskip
    Now we check   \eqref{bounded_1_eta_p}.
   
   Note that, by \eqref{estMRphi} and $\psi \in \mathcal H^{(n)}_{\mu}$ we have    
    \[\sup_{z\in \BB}\mu(z)|\mathscr B_{j}^n(\psi; \eta_p)(z)| <\infty.\] 
  By a similar proof to that of \eqref{eq_inf_A}, we also obtain.  
		\[\inf_{|w|>\lambda}\mathscr B^{n-}_{j,p} := \inf_{|w|>\lambda}\mu(w)\big|\mathscr B_{j}^n(\psi; \varphi_p)(w)\big|   >0\quad\text{for every $j= 0, 1, \ldots, n$}\]
		 and for some $\lambda \in (0, 1).$
Thus
\[D'_j := \frac{\sup_{z\in \BB}\mu(z)|\mathscr B_{j}^n(\psi; \eta_p)(z)|}{\inf_{|w|>\lambda}\mathscr B^{n-}_{j,p}} <\infty.\]

Therefore,
  by appliying Lemma \ref{lem_B_n}(b) to the functions $h_j(z') := \mu(z')\mathscr B_{j}^n(\psi(z'); R\varphi_p(z')),$ from the estimate \eqref{est_bounded_1}  we have
\[\aligned
\mu(z)&\big|\mathscr B^n_{j}(\psi; \eta_p)(z)\big|\Big\|\delta_{\eta_p(z)}^{\mathcal{H}^{(n+m-j)}_{\nu}}\Big\|  \\
&\le\frac{\sup_{z\in \BB}\mu(z)\big|\mathscr B^n_{j}(\psi; \eta_p)(z)|}{\inf_{|w|>\lambda}\mu(w)\big|\mathscr B^n_{j}(\psi; \varphi_p)(w)|}\mu(z')\big|\mathscr B^n_{j}(\psi; \varphi_p)(z')\Big\|\delta_{\varphi_p(z')}^{\mathcal{H}^{(n+m-j)}_{\nu}}\Big\|\\
&\le D'_j\mathscr B^n_{j,p}<\infty
\endaligned\]
for every $|z|>\lambda.$ On the other hand, it is obvious that
\[\sup_{|z|\le \lambda}\mu(z)\big|\mathscr B^n_{j}(\psi; \eta_p)(z)\big|\Big\|\delta_{\eta_p(z)}^{\mathcal{H}^{(n+m-j)}_{\nu}}\Big\|  <\infty.\]
Hence, \eqref{bounded_1_eta_p} is proved.

Thus, $W_{\psi,\eta}$ is bounded.

	Now, it is easy to check that $W_{\psi,\eta} = W_{\psi,\varphi}\circ C_\gamma.$  
 	Then by Lemma \ref{lem_prepare_3}, $W_{\psi,\varphi}$ is bounded, hence, 2) $\Rightarrow$ 1) is proved.

The proof of Theorem is completed.
\end{proof}

\begin{theorem}\label{thm_A2} Let $n, m \in\N_0.$  Assume that $\varphi  \in \widetilde S_p(\BB)$ for some $p\in \{{1,\ldots,N\}},$ such that such that the condition ($n+m, \mu$) satisfied.

Then, the  following are equivalent:  
\begin{enumerate}
	\item [\rm1)] $ W_{\psi,\varphi}:  \mathcal{H}^{(n)}_{\nu} \to \mathcal{H}^{(n+m)}_{\mu}$ is bounded;
	\item[\rm2)] $\psi, \psi\cdot\varphi^i_p \in \mathcal{H}^{(n+m)}_{\mu}$ for every  $i = 0,1,2 \ldots, $  and 
		\begin{equation}\label{est_B_4} 
  \mathscr B^{n+m}_{j,p} := \sup_{z\in \BB}\mu(z)\big|\mathscr B^{n+m}_{j}(\psi; \varphi_p)(z)\big|\Big\|\delta_{\varphi_p(z)}^{\mathcal{H}^{(n-j)}_{\nu}}\Big\|  <\infty\quad \text{for every}\ 0\le j \le  n.\end{equation}
\begin{equation}\label{est_B_5} 
  \mathscr B^{n+m}_{n+k,p} := \sup_{z\in \BB}\mu(z)\frac{\big|\mathscr B^{n+m}_{n+k}(\psi; \varphi_p)(z)\big|}{\nu(\varphi_p(z))(1-|\varphi_p(z)|^2)^k} <\infty\quad \text{for every}\ 1\le k \le  m.\end{equation}
\end{enumerate}
In this case  
\begin{equation}\label{norm_A2}
\|W_{\psi,\varphi}\| \asymp |\psi(0)|\Big\|\delta_{\varphi_p(0)}^{H^{(n)}_{\nu^p}}\Big\| + \sum_{j=0}^{n+m} \mathscr B^{n+m}_{j,p}. \end{equation}
 \end{theorem}
\begin{proof} First,   using the same argument as in the proof of Theorem 4.2 in \cite{Qu1}, we obtain \eqref{norm_A2} in the case where $ W_{\psi,\varphi}:  \mathcal{H}^{(n)}_{\nu} \to \mathcal{H}^{(n+m)}_{\mu}$ is bounded.

1) $\Rightarrow$ 2): It follows from Theorem 4.2 in \cite{Qu1} that $\psi, \psi\cdot\varphi^i_p  \in \mathcal{H}^{(n+m)}_{\mu}$ for every  $i = 0,1,2 \ldots, $ and 
  \[\mathscr B^{n+m}_{j} := \sup_{z\in \BB}\mu(z)\big|\mathscr B^{n+m}_{j}(\psi; \varphi)(z)\big|\Big\|\delta_{\varphi(z)}^{\mathcal{H}^{(n-j)}_{\nu}}\Big\|  <\infty \quad\text{for every $0\le j \le  n;$}\] 
  \[\mathscr B^{n+m}_{n+k} := \sup_{z\in \BB}\mu(z)\frac{\big|\mathscr B^{n+m}_{n+k}(\psi; \varphi)(z)\big|}{\nu^p(\varphi(z))(1-|\varphi(z)|^2)^k} <\infty\quad\text{for every $1\le k \le  m,$}\] 
    hence, \eqref{est_B_4} and \eqref{est_B_5} are true.
 
 2) $\Rightarrow$ 1):  As in the proof of Theorem \ref{thm_A1}, we also consider two cases.
 
  $\bullet$ The case  $\varphi(0) = 0.$
  
  As in the previous theorem, for $0\le j \le  n+m$ we have $D_j <\infty$ and for each $z \in \BB,$ $|\varphi(z)|>\lambda$ (hence, $|z|>\lambda$)  there exists $z' \in \BB,$ such that $|\varphi(z)| = |\varphi_p(z')|$ (hence, $|z'|>\lambda$). By an estimate in the proof of Theorem 4.2 in \cite{Qu1} we have
 \[\aligned
 \mu(z)|R^{(n+m)}W_{\psi,\varphi}(f)|  &\lesssim \sum_{j=0}^{n}  \mu(z)\big|\mathscr B^{n+m}_{j}(\psi; \varphi)(z)\big|\Big\|\delta_{\varphi(z)}^{\mathcal{H}^{(n-j)}_\nu}\Big\|\|f\|_{\mathcal{H}^{(n)}_\omega}  \\
 &\quad  + \sum_{k=1}^m\frac{\big|\mathscr B^{n+m}_{n+k}(\psi; \varphi)(z)\big|}{\nu(\varphi(z))(1-|\varphi(z)|^2)^k}\|f\|_{\mathcal{H}^{(n)}_\omega} \\
  &\le \sum_{j=0}^{n}   D_j\sup_{w\in \BB}\mu(w)\big|\mathscr B_{j}^n(\psi; \varphi_p)(w)\big|\Big\|\delta_{\varphi_p(w)}^{\mathcal{H}^{(n-j)}_\nu}\Big\| \|f\|_{\mathcal{H}^{(n)}_\nu}  \\
  &\quad  + \sum_{k=1}^mD_{n+k}\frac{\mu(z')\big|\mathscr B^{n+m}_{n+k}(\psi; \varphi_p)(z')\big|}{\nu(\varphi_p(z'))(1-|\varphi_p(z')|^2)^k}\|f\|_{\mathcal{H}^{(n)}_\omega} \\
    &\le \sum_{j=0}^{n}   D_j\sup_{w\in \BB}\mu(w)\big|\mathscr B_{j}^n(\psi; \varphi_p)(w)\big|\Big\|\delta_{\varphi_p(w)}^{\mathcal{H}^{(n-j)}_\nu}\Big\| \|f\|_{\mathcal{H}^{(n)}_\nu}  \\
      &\quad  + \sum_{k=1}^mD_{n+k}\sup_{w\in \BB}\frac{\mu(w)\big|\mathscr B^{n+m}_{n+k}(\psi; \varphi_p)(w)\big|}{\nu(\varphi_p(w))(1-|\varphi_p(w)|^2)^k}\|f\|_{\mathcal{H}^{(n)}_\nu} \\
  &\lesssim \sum_{j=0}^{n+m} \mathscr B^{n+m}_{j,p}\|f\|_{\mathcal{H}^{(n)}_\nu} \quad\text{ for every $z\in \BB.$}
 \endaligned\]
  Thus, $W_{\psi,\varphi}$ is bounded.
  
  \medskip
  $\bullet$ The case  $\varphi(\alpha) = 0,$ $\alpha \in \BB\setminus \{0\}.$
  
  Consider   $\gamma_\alpha \in Aut(\mathbb B)$ given as in the proof of the Theorem \ref{thm_A1},  at the same time, $\psi, \eta_p$ satisfy the condition ($n+m, \mu$) and
  \[\sup_{z\in \BB}\mu(z)\big|\mathscr B_{j}^n(\psi; \eta_p)(z)\big|\Big\|\delta_{\eta_p(z)}^{\mathcal{H}^{(n-j)}_\nu}\Big\|<\infty\quad\text{for $0\le j \le  n.$}\]

  It remain to check that
  \[
  \mathscr B^{n+m}_{n+k,\eta_p} := \sup_{z\in \BB}\mu(z)\frac{\big|\mathscr B^{n+m}_{n+k}(\psi; \eta_p)(z)\big|}{\nu(\eta_p(z))(1-|\eta_p(z)|^2)^k} <\infty\quad\text{for every $1\le k \le  m.$}\]

We have
\[\aligned
\mu(z)&\frac{\big|\mathscr B^{n+m}_{n+k}(\psi; \eta_p)(z)\big|}{\nu(\eta_p(z))(1-|\eta_p(z)|^2)^k} \\
&\le \frac{\frac{\mu(z)\big|\mathscr B^{n+m}_{n+k}(\psi; \eta_p)(z)\big|}{\nu(\varphi_p(z))(1-|\varphi_p(z)|^2)^k}}{\inf_{w\in \BB}\frac{\mu(w)\big|\mathscr B^{n+m}_{n+k}(\psi; \eta_p)(z)\big|}{\nu(\varphi_p(z))(1-|\varphi_p(z)|^2)^k}}\frac{\mu(z')\big|\mathscr B^{n+m}_{n+k}(\psi; \varphi_p)(z')\big|}{\nu(\varphi_p(z'))(1-|\varphi_p(z')|^2)^k} \\
&\le D_{n+k}\mathscr B^{n+m}_{n+k,p} <\infty.
\endaligned\]
Thus, $  \mathscr B^{n+m}_{n+k,\eta_p} <\infty.$
\end{proof}
Finally, we characterize the compactness   of weighted composition operator $W_{\psi,\varphi}$.

\begin{theorem}\label{thm_C1}  Assume that $\varphi  \in \widetilde S_p(\BB)$ for some $p\in \{{1,\ldots,N\}},$ such that   the condition ($n, \mu$) satisfied and there exists $n_0 \in \{0, \ldots,n+1\}$ such that
  \[I^{n+m-n_0+1}_\nu(1) <\infty = I^{n+m-n_0}_\nu(1).\]
   Then, the  following are equivalent:  
  \begin{enumerate}
	\item[\rm1)]   $ W_{\psi,\varphi}:  \mathcal{H}^{(n+m)}_{\nu} \to \mathcal{H}^{(n)}_{\mu}$ is compact;
	\item[\rm2)] $\psi, \psi\cdot\varphi^i_p  \in \mathcal{H}^{(n)}_{\mu}$ for every  $i = 0,1,2 \ldots, $  and for every $n_0\le j \le  n+1:$
		\begin{equation}\label{est_C_A} 
	 \lim_{r\to 1}\sup_{|\varphi_p(z)|>r}\mu(z)\big|\mathscr B^n_{j}(\psi; \varphi_p)(z)\big|\Big\|\delta_{\varphi_p(z)}^{\mathcal{H}^{(n+m-j)}_{\nu}}\Big\|= 0.\end{equation}
\end{enumerate}
  \end{theorem}
  \begin{proof} 1) $\Rightarrow$ 2): 
It follows from Theorem 5.2 in \cite{Qu1} that $\psi, \psi\cdot\varphi^i_p  \in \mathcal{H}^{(n)}_{\mu}$ for every  $i = 0,1,2 \ldots, $ and and for every $n_0\le j \le  n+1:$
	 \[\lim_{r\to 1}\sup_{|\varphi(z)|>r}\mu(z)\big|\mathscr B^n_{j}(\psi; \varphi)(z)\big|\Big\|\delta_{\varphi(z)}^{\mathcal{H}^{(n+m-j)}_{\nu^p}}\Big\|= 0,\] 
hence,  \eqref{est_C_A}  holds.

2) $\Rightarrow$ 1): As in the case of the boundedness, we also consider two cases.

 $\bullet$ The case  $\varphi(0) = 0.$

  It follows from the assumption 2) and Theorem \ref{thm_A1} that   $W_{\psi, \varphi}:  \mathcal{H}^{(n+m)}_{\nu} \to \mathcal{H}^{(n)}_{\mu}$     is bounded and it folows from $\psi, \psi\cdot\varphi^i_p  \in \mathcal{H}^{(n)}_{\mu}$ for every  $i = 0,1,2 \ldots, $ that \eqref{eq_B+-} holds for every $j=0, 1,\ldots, n.$

Note  first that, by \eqref{eq_est_parial_f}  and Remark  \ref{rmk_h}, in fact, with  an argument analogous to the estimate \eqref{es_final},  we can find $D_j >0$ such that for $|\varphi(z)| >\lambda$
 \[\aligned
 \mu(z)&|R^{(n)}W_{\psi,\varphi}(f)(z)| \lesssim \psi(0)f(\varphi(0))\\
 & \ +\sum_{j=0}^n D_j  \sup_{|\varphi_p(w)|>\lambda}\mu(w)\big|\mathscr B_{j}^n(\psi; \varphi_p)(w)\big|\Big\|\delta_{\varphi_p(w)}^{\mathcal{H}^{(n+m-j)}_\nu}\Big\| \|f\|_{\mathcal{H}^{(n+m)}_\nu}
 \endaligned\]
 for every $f\in \mathcal H^{(n+m)}_\nu.$
 
 Let $\{f_s\}_{s\ge1}$ be a bounded sequence in $\mathcal H^{(n+m)}_{\nu}$ converging to $0$ uniformly on compact subsets of $\BB$ and fix  $\varepsilon > 0.$ Then by Cauchy integral formula and Lemma \ref{lem_I_infty}, we can choose $s_0 \in \N$   such that for $s\ge s_0$ such that
$$ |f_s(\varphi(0))| < \frac{\varepsilon}{2\|\psi\|_{\mathcal H^{(n)}_\mu}}, \quad \sup _{z \in \BB}\bigg|\frac{\partial^j f_s(\varphi(z)) }{\partial z_{\vec{l}}}
 \bigg| < \frac{\varepsilon}{2n_0D_j\mathscr B_{j}^{n-}}$$ 
 for $j = 0, \ldots, i,$ and by the hypothesis there exists $\lambda > 0$ such that for every $n_0 \le j \le n+1$ and for $\lambda < |\varphi_p(z)| < 1,$   
$$   \mu(z)\big|\mathscr B_{j}^n(\psi; \varphi_p)(z)\big|\Big\|\delta_{\varphi_p(z)}^{\mathcal{H}^{(n+m-j)}_\nu}\Big\|   <\frac{\varepsilon}{2(n-n_0+2)D_jK}, $$ 
where $K := \sup_{s\ge1}\|f_s\|_{\mathcal H^{(n+m)}_\nu} <\infty.$ 
Then   for  every $ s\ge s_0 $ and $ |\varphi_p(z)| > \lambda, $ by Lemma \ref{lem_I_infty}, \eqref{eq_est_parial_f}, \eqref{eq_est_parial_f_more2}  and $\psi \in \mathcal H^{(n)}_\mu$ we have
\begin{equation}\label{est_compact1}
\aligned
 \mu(z)&|R^{(n)}W_{\psi,\varphi}(f_s)(z)| \lesssim \psi(0)|f_s(\varphi(0))| \\
 &\  +\sum_{j=0}^{n_0-1} D_j\bigg|\frac{\partial^j f_s(\varphi(z)) }{\partial z_{\vec{l}}}
 \bigg|  \sup_{|\varphi_p(w)|>\lambda}\mu(w)\big|\mathscr B_{j}^n(\psi; \varphi_p)(w)\big|  \\
  &\  +\sum_{j=n_0}^n D_j  \sup_{|\varphi_p(w)|>\lambda}\mu(w)\big|\mathscr B_{j}^n(\psi; \varphi_p)(w)\big|\Big\|\delta_{\varphi_p(w)}^{\mathcal{H}^{(n+m-j)}_\nu}\Big\| \|f_s\|_{\mathcal{H}^{(n+m)}_\nu} \\
 &  \le \|\psi\|_{\mathcal H^{(n)}_{\mu}}\frac{\varepsilon}{2\|\psi\|_{\mathcal H^{(n)}_\mu}} + \sum_{j=0}^{n_0-1} D_j \frac{\varepsilon \mathscr B_{j}^{n-}}{2n_0D_j\mathscr B_{j}^{n-}} + \sum_{j=0}^n D_j \frac{\varepsilon K}{2(n+1)D_jK} = \varepsilon.\\
 \endaligned\end{equation}
  On the other hand, since $\{f_s\}_{s\ge1}$  converges to $0$ uniformly on compact subsets of $\BB,$ by Cauchy integral formula again,  it is clear that
  \[\sup_{|\varphi_p(z)| \le \lambda}\bigg|\frac{\partial^j f_s(\varphi(z)) }{\partial z_{\vec{l}}}
 \bigg|  \to 0 \quad \text{as}\ s \to \infty\]
  for every $j=0, 1, \ldots, n.$
   Then, by \eqref{eq_est_parial_f_more2}, \eqref{eq_B+-}  and   with the estimate as above, we have
 \begin{equation}\label{est_compact2}
	\begin{aligned}
		\sup_{|\varphi(z)|\le \lambda}&\mu(z)|R^{(n)}W_{\psi,\varphi}(f_s)(z)| \\
		&\lesssim \|\psi\|_{\mathcal H^{(n)}_\mu}|f_s(0)|   + \sum_{j=0}^n\mathscr B_{j}^{n-}\sup_{|\varphi_p(z)| \le \lambda}\bigg|\frac{\partial^j f_s(\varphi(z)) }{\partial z_{\vec{l}}}
 \bigg|   \to 0  
	\end{aligned}
\end{equation}
as $s \to \infty.$ Therefore, it follows from Lemma 3.6 in \cite{Qu1} and \eqref{est_compact1}, \eqref{est_compact2} that $ W_{\psi,\varphi} $ is compact.

 $\bullet$ The case  $\varphi(\alpha) = 0,$ $\alpha \in \BB\setminus \{0\}.$
 
 Similar to the reasoning in the proof of Theorem \ref{thm_A1}, we can easily show that  
	 \[\lim_{r\to 1}\sup_{|\eta_p(z)|>r}\mu(z)\big|\mathscr B^n_{j}(\psi; \eta_p)(z)\big|\Big\|\delta_{\eta_p(z)}^{\mathcal{H}^{(n+m-j)}_{\nu}}\Big\|= 0\]
 holds when \eqref{est_C_A} occurs, and thus, the theorem is proved.
  \end{proof}
 
 Now, using Theorem 5.3 in \cite{Qu1} and reasoning as in the proof of the above theorem, we easily obtain the following result. the proofs of which will be omitted.

  \begin{theorem}\label{thm_C2}   Assume that $\varphi  \in \widetilde S_p(\BB)$ for some $p\in \{{1,\ldots,N\}},$ such that  the condition ($n+m, \mu$) satisfied   
				and there exists $n_0 \in \{0, \ldots,n+1\}$    such that
  \[I^{n-n_0+1}_\nu(1) <\infty = I^{n-n_0}_\nu(1).\]
   Then, the  following are equivalent:  
  \begin{enumerate}
	\item[\rm1)]   $ W_{\psi,\varphi}:  \mathcal{H}^{(n)}_{\nu} \to \mathcal{H}^{(n+m)}_{\mu}$ is compact;
	\item[\rm2)] $\psi, \psi\cdot\varphi^i_p  \in \mathcal{H}^{(n+m)}_{\mu}$ for every  $i = 0,1,2 \ldots, $   
   \[\lim_{r\to 1}\sup_{|\varphi(z)|>r}\mu(z)\big|\mathscr B^{n+m}_{j,p}(\psi; \varphi_p)(z)\big|\Big\|\delta_{\varphi_p(z)}^{\mathcal{H}^{(n-j)}_{\nu}}\Big\| =0\quad\text{   for every $j = 0, 1,\ldots, n;$}\] 
  \[\lim_{r\to 1}\sup_{|\varphi(z)|>r}\mu(z)\frac{\big|\mathscr B^{n+m}_{n+k,p}(\psi; \varphi_p)(z)\big|}{\nu(\varphi_p(z))(1-|\varphi_p(z)|^2)^k} =0\quad\text{for every $1\le k \le  m.$}\] 
\end{enumerate}
  \end{theorem}

\section*{Data availability}  The paper does not use any data set. 
\section*{Declarations}
\begin{itemize}
\item Conflict of interest: The author has no conflict of interests related to this work.  
\item Funding: The author did not receive support from any organization for this work. 
\end{itemize}



\begin{thebibliography}{[ASMR]}
   \bibitem{CH} Colonna, F.,   Hmidouch, N.:  Weighted composition operators on iterated weighted-type Banach spaces of analytic functions, Complex Anal.   Oper. Theory {\bf 13},  1989-2016  (2019).
   \bibitem{Gi} Girela D., On Bloch functions and gap series, Publ. Matem\`atiques,  35, 403-427,  (1991)
   \bibitem{HJ}  Huang, C-S.,   Jiang, Z-J.:  On a sum of more complex product-type operators from
Bloch-type spaces to the weighted-type spaces, Axioms 12, 566, (2023) 
https://doi.org/10.3390/axioms12060566
\bibitem{LDQ1}  Lam, L. V., Dai, N. V., Quang, T. T.,     \textit{Some new characterizations of the weighted composition operators between Bloch-type spaces,}     European J. Math., 11(25) (2025). https://doi.org/10.1007/s40879-025-00817-w
\bibitem{LDQ2} Lam, L. V., Dai, N. V., Quang, T. T., \textit{Some new characterizations of the weighted composition operators from Bloch-type spaces to Zygmund-type spaces,} sub. to Rendiconti del Circolo Matematico di Palermo Series 2.
\bibitem{LQ} Lam, L. V., Quang, T. T., \textit{On   the boundedness and compactness of  extended Ces\`aro   composition operators between  weighted Bloch-type spaces,}   Hacet. J. Math. Stat.,  53 (4) (2024), 897-914.
 \bibitem{Qu1} Quang, T. T., \textit{Weighted composition operators on weighted-type high-order growth spaces on the unit ball}, J. Math. Anal. Appl.,  547 (2025), 129266, 
 https://doi.org/10.1016/j.jmaa.2025.129266
    \bibitem{Ya}  Yamashita, S.,  Gap series and $\alpha$-Bloch functions, Yokohama Math. J. 28, 31-36 (1980).
     \bibitem{Zh} Zhu, K.:  Spaces of holomorphic functions in the unit ball, Graduate Texts in Mathematics, vol. \textbf{226}, Springer-Verlag, New York, (2005)
     \end{thebibliography}

\end{document}